\documentclass[3p,times]{elsarticle}

\usepackage{amsfonts}
\usepackage{amsfonts,amsmath,amssymb}
\usepackage{mathrsfs,mathtools,stmaryrd,wasysym}
\usepackage{enumerate}
\usepackage{xspace} 
\usepackage{esint}
\usepackage{hyperref}
\DeclareMathAlphabet{\mathpzc}{OT1}{pzc}{m}{it}

\newcommand{\R}{\mathbb{R}}

\newcommand{\calX}{\mathcal{X}}
\newcommand{\calY}{\mathcal{Y}}

\newcommand{\bu}{\mathbf{u}}
\newcommand{\bv}{\mathbf{v}}

\newcommand{\GRAD}{\nabla}
\DeclareMathOperator{\DIV}{div}

\newcommand{\ie}{i.e.,\@\xspace}

\newcommand{\bF}{{\mathbf{F}}}
\newcommand{\bL}{{\mathbf{L}}}
\newcommand{\bW}{{\mathbf{W}}}
\newcommand{\bH}{{\mathbf{H}}}

\newcommand{\bef}{{\mathbf{f}}}
\newcommand{\be}{{\mathbf{e}}}
\newcommand{\bX}{{\mathbf{X}}}

\newcommand{\bw}{{\mathbf{w}}}

\newcommand{\pe}{{\mathsf{p}}}
\newcommand{\ue}{{\mathsf{u}}}

\newcommand{\calS}{\mathcal{S}}
\newcommand{\NL}{\mathcal{NL}}
\newcommand{\calF}{\mathcal{F}}
\newcommand{\calT}{\mathcal{T}}
\newcommand{\calB}{\mathcal{B}}

\newcommand{\vare}{\varepsilon}

\newcommand{\TheTitle}{A weighted setting for the stationary Navier Stokes equations under singular forcing}

\newtheorem{lemma}{Lemma}
\newtheorem{theorem}{Theorem}
\newtheorem{proposition}{Proposition}
\newtheorem{corollary}{Corollary}
\newtheorem{definition}{Definition}
\newproof{proof}{Proof}
\newdefinition{remark}{Remark}

\journal{arXiv}

\begin{document}

\begin{frontmatter}
\title{{\TheTitle}\tnoteref{t1}}
\tnotetext[t1]{EO has been partially supported by CONICYT through FONDECYT project 11180193. AJS has been partially supported by NSF grant DMS-1720213.}
\author[1]{Enrique Ot\'arola}
\ead{enrique.otarola@usm.cl}
\ead[url]{http://eotarola.mat.utfsm.cl/}

\author[2]{Abner J.~Salgado}
\ead{asalgad1@utk.edu}
\ead[url]{http://www.math.utk.edu/~abnersg}

\address[1]{Departamento de Matem\'atica, Universidad T\'ecnica Federico Santa Mar\'ia, Valpara\'iso, Chile.}
\address[2]{Department of Mathematics, University of Tennessee, Knoxville, TN 37996, USA.}

\begin{abstract}
In two dimensions, we show existence of solutions to the stationary Navier Stokes equations on weighted spaces $\bH^1_0(\omega,\Omega) \times L^2(\omega,\Omega)$, where the weight belongs to the Muckenhoupt class $A_2$. We show how this theory can be applied to obtain a priori error estimates for approximations of the solution to the Navier Stokes problem with singular sources.
\end{abstract}


\end{frontmatter}

\section{Introduction}
\label{sec:intro}

Let $\Omega \subset \R^2$ be an open and bounded domain with Lipschitz boundary $\partial \Omega$. In this work we will be interested in developing an existence and approximation theory for the Navier Stokes problem
\begin{equation}
\label{eq:NSEStrong}
    -\nu \Delta \ue + (\ue \cdot \GRAD) \ue + \GRAD \pe = \bef, \text{ in } \Omega, \qquad
    \DIV \ue = 0, \text{ in } \Omega, \qquad
    \ue = 0, \text{ on } \partial\Omega,
\end{equation}
in the case where the forcing term $\bef$ is singular. Here, the unknowns are $\ue$ and $\pe$, the velocity and pressure, respectively. The data are the forcing term $\bef$ and the kinematic viscosity $\nu > 0$.

Essentially, by introducing a weight, we can allow for forces such that $\bef \notin \bH^{-1}(\Omega)$. In particular, our theory will allow the following particular examples. For a fixed $\bF \in \R^d$, we can set $\bef = \bF \delta_z$, where $\delta_z$ denotes the Dirac delta supported at the interior point $z \in \Omega$. Similarly, if $\Gamma$ denotes a smooth closed curve contained in $\Omega$, we can allow the components of $\bef$ to be measures supported in $\Gamma$.

We must remark that the study of \eqref{eq:NSEStrong} in a non energy setting is not a new idea. For instance, \cite{MR753154} assumes that the domain is smooth, but deals with rough forcings $\bef \in \bW^{-1,p}(\Omega)$ with $p \in (d/2,2)$ ($d$ is the space dimension). A nonenergy setting and weights are commonly used in the exterior problem for \eqref{eq:NSEStrong}; see, for instance, \cite{MR1407336,MR2551492}. Closely related to our work is \cite{MR2272870}, where an existence and uniqueness theory in $C^{1,1}$ domains is developed over weighted spaces, and under the assumption that the data is small. Our main novelty here is that we only assume tha domain to be Lipschitz, and we provide existence for arbitrary data. In addition, when the domain is a convex polygon, we show convergence of a finite element scheme.

\section{Weak formulation}
\label{sec:weak}

We begin by recalling that, if $\bv$ is sufficiently smooth and solenoidal, then the convective term can be rewritten as
$
  (\bv \cdot \GRAD) \bv = \DIV ( \bv \otimes \bv).
$
This will be used in the weak formulation of \eqref{eq:NSEStrong}.

% \begin{proposition}[skew symmetry]
% Assume that the function $\bv$ is sufficiently smooth and solenoidal, then
% \[
%   (\bv \cdot \GRAD) \bv = \DIV ( \bv \otimes \bv).
% \]
% \end{proposition}
% \begin{proof}
% The result follows from an immediate computation. If $k \in \{1,2\}$, then
% \[
%   \left[ (\bv \cdot \GRAD) \bv \right]_k = \sum_{i=1}^2 \bv_i \partial_i \bv_k = \sum_{i=1}^2 \partial_i ( \bv_i \bv_k) - \sum_{i=1}^2 \partial_i \bv_i \bv_k = \DIV(\bv \otimes \bv)_k - \DIV \bv \bv_k.
% \]
% The fact that $\bv$ is solenoidal allows us to conclude.
% \qed\end{proof}

Let $\omega$ be a weight. We define $\calX = \bH^1_0(\omega,\Omega) \times L^2(\omega,\Omega)/\R$ and $\calY = \bH^1_0(\omega^{-1},\Omega) \times L^2(\omega^{-1},\Omega)/\R$. We propose the following weak formulation of problem \eqref{eq:NSEStrong}. Given $\bef \in ( \bH^1_0(\omega^{-1},\Omega))'$, find $(\ue,\pe) \in \calX$ such that
\begin{equation}
\label{eq:NSEVar}
  \int_\Omega \left( \nu \GRAD  \ue : \GRAD \bv - \ue \otimes \ue : \GRAD \bv - \pe \DIV \bv\right)  = \langle \bef, \bv \rangle \ \forall \bv \in \bH^1_0(\omega^{-1},\Omega), \quad
  \int_\Omega \DIV \ue q = 0 \ \forall q \in L^2(\omega^{-1},\Omega)/\R.
\end{equation}
Here and in what follows, by $\langle \cdot, \cdot\rangle$ we denote a duality pairing.

An application of the Cauchy--Schwarz inequality reveals that, for $(\ue,\pe) \in \calX$ and $(\bv,q) \in \calY$, the terms
$\nu \int_\Omega \GRAD \ue : \GRAD \bv$, $\int_\Omega \pe \DIV \bv$, and $\int_\Omega q \DIV \ue$
% \[
%   \nu \int_\Omega \GRAD \ue : \GRAD \bv, \qquad \int_\Omega \pe \DIV \bv, \qquad \int_\Omega q \DIV \ue
% \]
are bounded. The following result provides suitable assumptions on the weight, that guarantee the boundedness of the convective term.

\begin{lemma}[boundedness of convection]
\label{lem:convbdd}
If $\omega \in A_2$, then, for $\bv \in \bH^1_0(\omega,\Omega)$ and $\bw \in \bH^1_0(\omega^{-1},\Omega)$, we have
\[
  \left| \int_\Omega \bv \otimes \bv : \GRAD \bw \right| \leq C \| \GRAD \bv \|_{\bL^2(\omega,\Omega)}^2 \| \GRAD \bw \|_{\bL^2(\omega^{-1},\Omega)}.
\]
\end{lemma}
\begin{proof}
According to \cite[Theorem 1.3]{MR643158}, since we are in two dimensions and $\omega \in A_2$, we have $H^1_0(\omega,\Omega) \hookrightarrow L^4(\omega,\Omega)$. We can thus write
\[
  \left| \int_\Omega \bv \otimes \bv : \GRAD \bw \right| = \left| \int_\Omega (\omega^{\frac{1}{4}}\bv) \otimes (\omega^{\frac{1}{4}}\bv) : (\omega^{-\frac{1}{2}}\GRAD \bw) \right| \leq \| \bv \|_{\bL^4(\omega,\Omega)}^2 \| \GRAD \bw \|_{\bL^2(\omega^{-1},\Omega)}.
\]
Conclude by using the aforementioned embedding.
\qed\end{proof}

\begin{remark}[two dimensions]
\label{rem:2d}
It is in this result that the assumption $D \subset \R^2$ plays an essential r\^ole. Indeed, \cite[Theorem 1.3]{MR643158} states that, if $D \subset \R^d$ and $\varpi \in A_2$, then there is $\delta>0$ such that $H^1_0(\varpi,D) \hookrightarrow L^{2k}(\varpi,D)$ for $k \in [1,d/(d-1) + \delta]$. In three dimensions then, we only have $H^1_0(\varpi,D) \hookrightarrow L^3(\varpi,D)$ and, at least with this approach, we cannot show boundedness of the convective term.
\end{remark}

The next definition, that is inspired by \cite[Definition 2.5]{MR1601373}, will be of importance for the analysis that follows.
\begin{definition}[class $A_p(D)$]
\label{def:ApOmega}
Let $D \subset \R^d$ be a Lipschitz domain. For $p \in (1, \infty)$ we say that $\omega \in A_p$ belongs to $A_p(D)$ if there is an open set $\mathcal{G} \subset D$, and positive constants $\varepsilon>0$ and $\omega_l>0$, such that:
%\begin{enumerate}[(a)]
%\item
$\{ x \in \Omega: \mathrm{dist}(x,\partial D)< \varepsilon\} \subset \mathcal{G}$,
%\item 
$\omega \in C(\bar {\mathcal{G}})$, and $\omega_l \leq \omega(x)$ for all $x \in \bar{\mathcal{G}}$.
%\end{enumerate}
\end{definition}

\section{Existence of solutions}
\label{sec:existence}

Having defined our weak formulation, here we show existence of solutions and, under a smallness assumption on the data, their uniqueness. To do so, let us define the mappings $\calS: \calX \to \calY'$, $\NL: \calX \to \calY'$, and $\calF \in \calY'$ by
\[
  \langle \calS(\bu,p), (\bv,q) \rangle = \int_\Omega \left(\GRAD \bu : \GRAD \bv - p \DIV \bv -\DIV \bu q\right), \quad
  \langle \NL(\bu,p), (\bv,q) \rangle = -\int_\Omega \bu \otimes \bu : \GRAD \bv, \quad
  \langle \calF, (\bv,q) \rangle = \langle \bef, \bv \rangle.
\]
Notice that \eqref{eq:NSEVar} can be equivalently written as the following operator equation in $\calY'$
\[
  \calS(\nu \ue,p) + \NL(\ue,p) = \calF.
\]

As we have shown above, $\calS$ is a bounded linear operator. Moreover, \cite[Theorem 17]{OS:17infsup} shows that, if $\Omega$ is Lipschitz and $\omega \in A_2(\Omega)$, then this operator has a bounded inverse. This allows us to define the operator $\calT: \calX \to \calX$ via
\[
  (\nu\bu,p) = \calT(\bw,r) = \calS^{-1} \left[ \calF - \NL(\bw,r) \right].
\]
Therefore, showing existence of a solution amounts to finding a fixed point of the mapping $\calT$. We will show existence and uniqueness for sufficiently small data, and existence for general data.

\subsection{Existence and uniqueness for small data}
\label{sub:smalldata}

Let us first show, via a contraction argument, that provided the problem data is sufficiently small we have existence and uniqueness of solutions. Our contraction argument is rather standard, see for instance \cite[Theorem 3.1]{MR2413675} and \cite[Theorem 5.6]{MR2272870}. The main novelty in our approach seems to be the fact that, by restricting the weight to $A_2(\Omega)$, we allow the domain to be merely Lipschitz. We begin by defining, for $A>0$,
\begin{equation}
\label{eq:defofcalB}
  \calB_A = \left\{ \bw \in \bH^1_0(\omega,\Omega): \DIV \bw = 0,  \ \| \GRAD \bw \|_{\bL^2(\omega,\Omega)} \leq A \right\}.
\end{equation}

In what follows, by $\|\calS^{-1} \|$ we denote the $\calY' \to \calX$ norm of $\calS^{-1}$, and by $C_{4 \to 2}$ we denote the constant in the embedding $\bH_0^1(\omega,\Omega) \hookrightarrow \bL^4(\omega,\Omega)$ which was used in Lemma~\ref{lem:convbdd}.

\begin{proposition}[contraction]
\label{prop:contraction}
Let $\Omega$ be Lipschitz and $\omega \in A_2(\Omega)$. Assume that the forcing term $\bef$ is sufficiently small, or the viscosity is sufficiently large, so that
\begin{equation}
\label{eq:smallness}
  \frac{C_{4\to2}^2 \| \calS^{-1} \|^2  \| \bef \|_{\bH^{-1}(\omega,\Omega)}}{\nu^2} < \frac16.
\end{equation}
Define
$
  A = \tfrac\nu{3C_{4\to2}^2\| \calS^{-1} \|}
$.
With these assumptions, the mapping $\calT_1: \bH^1_0(\omega,\Omega) \to \bH^1_0(\omega,\Omega)$ defined as 
$
  \bw \mapsto \frac1\nu \Pr \calT(\bw,0)
$,
where $\Pr : \calX \to \bH^1_0(\omega,\Omega)$ is the projection onto the velocity component, maps $\calB_A$ to itself and it is a contraction in it. 
\end{proposition}
\begin{proof}
Note, first of all, that the assumptions on the forcing term and viscosity can be summarized as
\[
  \frac{\| \calS^{-1} \| \ \| \bef \|_{\bH^{-1}(\omega,\Omega)}}{\nu} < \frac{A}2.
\]

Let us now show that $\calT_1$ maps $\calB_A$ to itself. Observe that, by definition of the mapping $\calT$, we have that, if $\bv = \calT_1(\bw)$, then $\DIV \bv = 0$ and
\begin{equation*}
  \| \GRAD \bv \|_{\bL^2(\omega,\Omega)} \leq \frac{\| \calS^{-1} \| \ \| \bef\|_{\bH^{-1}(\omega,\Omega)}}{\nu} +  \frac{ C_{4 \to 2}^2 \| \calS^{-1} \| \ \| \GRAD \bw \|_{\bL^2(\omega,\Omega)}^2}\nu 
    < \frac{A}2 + \frac{A}3 = \frac{5A}6,
\end{equation*}
where we used the data assumptions and the fact that $\bw \in \calB_A$.

We now show that $\calT_1$ is a contraction. For that, let $\bw_1,\bw_2 \in \calB_A$ and $\bv_i = \calT_1(\bw_i)$, for $i = 1,2$, respectively. Then, we have that
\begin{multline*}
  \| \GRAD(\bv_1 - \bv_2) \|_{\bL^2(\omega,\Omega)} \leq 
    \frac{C_{4\to2}^2 \| \calS^{-1} \|}\nu \left( \|\GRAD \bw_1 \|_{\bL^2(\omega,\Omega)} + \|\GRAD \bw_2 \|_{\bL^2(\omega,\Omega)} \right) \|\GRAD (\bw_1 - \bw_2) \|_{\bL^2(\omega,\Omega)}  \\
    \leq \frac{C_{4\to2}^2 \| \calS^{-1} \|}\nu \frac{2\nu}{3 C_{4\to2}^2 \| \calS^{-1} \|} \|\GRAD(\bw_1 - \bw_2) \|_{\bL^4(\omega,\Omega)}
    =
  \frac23 \|\GRAD(\bw_1 - \bw_2) \|_{\bL^4(\omega,\Omega)}.
\end{multline*}
This concludes the proof.
\qed\end{proof}

From this result we immediately obtain existence and uniqueness for small data.

\begin{corollary}[existence and uniqueness]
\label{cor:corContractionSmall}
Let $\Omega$ be Lipschitz and $\omega \in A_2(\Omega)$. Assume that the forcing term $\bef$ is sufficiently small, or the viscosity is sufficiently large, so that \eqref{eq:smallness} holds. In this setting, there is a unique solution of \eqref{eq:NSEVar}. Moreover, this solution satisfies the estimate
\[
  \| \GRAD \ue \|_{\bL^2(\omega,\Omega)} \leq \frac32 \frac{\| \calS^{-1}  \| \ \| \bef \|_{\bH^{-1}(\omega,\Omega)}}{\nu}.
\]
\end{corollary}
\begin{proof}
By assumption the mapping $\calT_1$, defined in Proposition~\ref{prop:contraction} has a unique fixed point $\ue \in \calB_A$. From this, by using the existence of a right inverse of the divergence operator over $A_2$--weighted spaces \cite[Theorem 3.1]{MR2731700}, existence and uniqueness of the pressure $\pe$ follows as well.

To obtain the claimed estimate, we use the fact that this is a fixed point of $\calT_1$. Indeed, from this it follows that
\[
  \| \GRAD \ue \|_{\bL^2(\omega,\Omega)} \leq \frac{ \| \calS^{-1} \| \ \| \bef \|_{\bH^{-1}(\omega,\Omega)}}{\nu} 
    + \frac{ C_{4\to2}^2\| \calS^{-1} \| }{\nu} A \| \GRAD \ue \|_{\bL^2(\omega,\Omega)},
\]
where we used that $\ue \in \calB_A$. Using the value of $A$ defined in Proposition~\ref{prop:contraction} we have
\[
 \frac{C_{4\to2}^2\| \calS^{-1} \| 
%  \, \| \bef \|_{\bH^{-1}(\omega,\Omega)}
  }{\nu} A = \frac{C_{4\to2}^2\| \calS^{-1} \| 
%  \, \| \bef \|_{\bH^{-1}(\omega,\Omega)}
  }{\nu} \frac\nu{3C_{4\to2}^2\| \calS^{-1} \|} = 1/3,
\]
from which the result follows.
\qed\end{proof}

\subsection{Existence for general data}

We now show existence of solutions without smallness conditions. As in the energy setting, we do not say anything about uniqueness of solutions. We begin with a series of preparatory steps. Our first result is about compact embedding between weighted spaces. To state it, we must assume that the \emph{set of singularities} $\mathbf{S}_{\textup{sing}}(\omega)$, as defined in \cite[Definition 4.2]{MR2797702}, is compactly contained in $\Omega$.

\begin{proposition}[compact embedding]
\label{prop:compactembedding}
Let $\omega \in A_2$ and $\mathbf{S}_{\textup{sing}}(\omega) \Subset \Omega$. In this setting, the embedding $H^1_0(\omega,\Omega) \hookrightarrow L^4(\omega,\Omega)$ is compact.
\end{proposition}
\begin{proof}
The result follows from \cite[Theorem 4.12]{MR2797702}. To see this, we set, in the notation of that paper, $n=2$, $s_1 = 1$, $p_1 = q_1 = 2$, $s_2 = 0$, $p_2 =4$, $q_2 = 2$, and $A = F$, that is, we work in the weighted Triebel-Lizorkin scale. Notice that the open ended property of $A_2$ implies that $r_{\omega} = \inf \{ r \geq 1: \omega \in A_r \} < 2$. Thus, we have
\[
\delta = s_1 - \frac{n}{p_1} - s_2 + \frac{n}{p_2} = \frac{1}{2},
\qquad (r_{\omega} - 1)\left(\frac{n}{p_1} - \frac{n}{p_2}\right) =  \frac{1}{2} (r_{\omega} - 1) < \frac{1}{2}.
\]
The conclusion of \cite[Theorem 4.12]{MR2797702}, together with \cite[(2.15)]{MR2797702}, then states that the embedding $H^1(\omega,\Omega) = F_{2,2}^1(\omega,\Omega) \hookrightarrow F_{4,2}^0(\omega,\Omega) = L^4(\omega,\Omega)$ is compact.
\qed\end{proof}

We remark that, for $\omega \in A_2(\Omega)$ the assumption that $\mathbf{S}_{\textup{sing}}(\omega) \Subset \Omega$ is automatically satisfied.

\begin{corollary}[$\NL$ is compact]
In the setting of Proposition~\ref{prop:compactembedding}, the operator $\NL$ is bounded and compact.
\end{corollary}
\begin{proof}
Boundedness of this mapping follows from Lemma~\ref{lem:convbdd}.

To show compactness, let $(\bw_n,r_n) \rightharpoonup (\bw,r)$ in $\calX$, so that $\bw_n \rightharpoonup \bw $ in $\bH^1_0(\omega,\Omega)$ and, by the compact embedding of Proposition~\ref{prop:compactembedding}, $\bw_n \to\bw$ in $\bL^4(\omega,\Omega)$. Now
\begin{align*}
  \left|\langle \NL(\bw_n,r_n) - \NL(\bw,r), (\bv,q) \rangle\right| & = \left|\int_\Omega (\bw \otimes \bw - \bw_n \otimes \bw_n) :\GRAD \bv\right|
    \leq  \left|\int_\Omega \bw \otimes(\bw - \bw_n) : \GRAD \bv\right| + \left|\int_\Omega (\bw - \bw_n) \otimes \bw_n : \GRAD \bv\right| \\
   & \leq 
   \left(  \| \bw \|_{\bL^4(\omega,\Omega)} + \| \bw _n \|_{\bL^4(\omega,\Omega)} \right) \| \bw_n - \bw \|_{\bL^4(\omega,\Omega)} \| \GRAD \bv \|_{\bL^2(\omega^{-1},\Omega)}.
\end{align*}
This shows that
\[
  \left\| \NL(\bw_n,r_n) - \NL(\bw,r) \right\|_{\calY'} \to 0,
\]
and the compactness follows.
\qed\end{proof}

\begin{corollary}[$\calT$ is compact]
\label{cor:NLcompact}
In the setting of Proposition~\ref{prop:compactembedding}, the mapping $\calT$ is compact.
\end{corollary}
\begin{proof}
This is immediate upon noting that $\calS^{-1}$ is continuous and $\NL$ is compact.
\qed\end{proof}

We are now in position to show our existence result.

\begin{theorem}[existence]
\label{thm:existence}
Let $\Omega$ be Lipschitz and $\omega \in A_2(\Omega)$. For every $\nu >0$ and $\bef \in ( \bH^1_0(\omega^{-1},\Omega) )'$, problem \eqref{eq:NSEVar} has at least one solution $(\ue,\pe) \in \calX$, which satisfies
\[
  \| (\ue,\pe)\|_\calX \leq C \| \calF \|_{\calY'}.
\]
\end{theorem}
\begin{proof}
We will show existence by showing that $\calT$ has a fixed point. Notice, first, that if $\calF = 0$ then we are in the setting of Corollary~\ref{cor:corContractionSmall} so that the only solution to the homogeneous problem is $(\ue,\pe) = (\boldsymbol0,0)$.

To show existence for general data we will invoke Schaefer's fixed point theorem \cite[Theorem 9.2.4]{MR2597943}. Let $\lambda \in (0,1]$ and assume that $(\ue,\pe) \in \calX$ is such that
\[
  (\nu \ue,\pe) = \lambda \calT(\ue,\pe) = \lambda \calS^{-1} \left[ \calF - \NL(\ue,\pe) \right].
\]
Since $\calS^{-1}$ is bounded, we have
\begin{equation*}
  \| (\nu \ue,\pe) \|_\calX \leq \lambda \| \calS^{-1} \calF \|_\calX + \lambda \| \calS^{-1} \NL(\ue,\pe) \|_\calX \leq C \| \calF \|_{\calY'} + \| \NL(\ue,\pe) \|_{\calY'} 
  \leq C \left( \| \calF \|_{\calY'} + \| \ue \|_{\bL^4(\omega,\Omega)}^2 \right),
\end{equation*}
where in the last step we used the boundedness of $\NL$ as shown in Lemma~\ref{lem:convbdd}. 
%Using the compact embedding of Proposition~\ref{prop:compactembedding}, we then obtain
%\[
%  \| (\ue,\pe) \|_\calX \leq C_1 \| \calF \|_{\calY'} + C_2 \| \ue \|_{\bL^4(\omega,\Omega)}.
%\]

We claim now that
\begin{equation}
\label{eq:ADN}
  \| (\ue,\pe) \|_\calX \leq C \| \calF \|_{\calY'},
\end{equation}
for if this is the case, then we can conclude the existence of a fixed point for the compact operator $\calT$, which is equivalent to existence of a solution for \eqref{eq:NSEVar}.

Showing \eqref{eq:ADN} can be carried out by the usual Avron--Douglis--Nirenberg contradiction. If the inequality is false, there is $\{(\ue_n,\pe_n)\}_{n \in \mathbb{N}}$ with $\| (\ue_n,\pe_n) \|_\calX = 1$ such that $\calF_n = \calS(\nu \ue_n,\pe_n) + \NL(\ue_n,\pe_n)$ satisfies $\| \calF_n \|_{\calY'} \to 0$. There is then a (not relabeled) subsequence $(\ue_n,\pe_n)  \rightharpoonup (\ue,\pe)$ in $\calX$ and $\ue_n \to \ue$ in $\bL^4(\omega,\Omega)$. Using uniqueness of solutions for the homogeneous problems we obtain that $\ue = \boldsymbol0$. However, this implies that
\[
  1 \leq  C_1 \| \calF_n \|_{\calY'} + C_2 \| \ue_n \|_{\bL^4(\omega,\Omega)}^2 \to 0,
\]
which is a contradiction.
\qed\end{proof}

\section{Discretization}
\label{sec:Discretization}

We now propose a finite element scheme to approximate the solution of \eqref{eq:NSEVar}. In what follows we will assume that $\Omega$ is a convex polygon, and that the weight is such that $\omega \in A_2(\Omega)$ and either $\omega$ or $\omega^{-1}$ belong to $A_1$. We let $\bX_h \times M_h \subset \calX$ be an admissible pair of finite element spaces, in the sense that they satisfy all the assumptions of \cite{DuranOtarolaAJS}. In this setting the Stokes projection onto $\bX_h \times M_h$ is stable in $\bH^1_0(\omega,\Omega) \times L^2(\omega,\Omega)$; see \cite[Theorem 4.1]{DuranOtarolaAJS}. This means that $\calS_h$, the discrete version of $\calS$, is a bounded linear operator whose inverse $\calS_h^{-1}$ is bounded uniformly with respect to $h$. We will make use of this fact in the error analysis.

We now define the discrete problem as: Find $(\bu_h,p_h) \in \bX_h \times M_h$ such that
\begin{equation}
\label{eq:Stokesh}
  \int_\Omega \left(\nu\GRAD \bu_h : \GRAD \bv_h - \bu_h \otimes \bu_h : \GRAD \bv_h - p_h \DIV \bv_h \right)= \langle \bef, \bv_h \rangle \ \forall \bv_h \in \bX_h, \quad
  \int_\Omega \DIV \bu_h q_h = 0 \ \forall q_h \in M_h.
\end{equation}

\subsection{Discretization for small data}
\label{sub:Discr}

We follow \cite[Chapter IV.3.1]{MR548867}, with suitable modifications to take into account that we are not in an energy framework anymore, \ie while setting $\bv_h = \bu_h$ is allowed, it does not lead to suitable estimates. We begin with the following existence and uniqueness result.

\begin{corollary}[existence and uniqueness]
Assume that either $\bef$ is sufficiently small or $\nu$ sufficiently large so that \eqref{eq:smallness} with $\calS^{-1}$ replaced by $\calS_h^{-1}$ holds. Then there is a unique $(\bu_h,p_h) \in X_h \times M_h$ that solves \eqref{eq:Stokesh}. Moreover, we have an estimate similar to that of Corollary~\ref{cor:corContractionSmall}. 
\end{corollary}
\begin{proof}
We repeat the proofs of Proposition~\ref{prop:contraction} and Corollary~\ref{cor:corContractionSmall}. The only point worth mentioning is that, instead of $\calS^{-1}$ we use the inverse of the discrete Stokes operator which, as we have previously stated, is uniformly bounded with respect to $h$.
\qed\end{proof}

With these results at hand we can obtain an error estimate.

\begin{theorem}[error estimate]
Assume that $\bef$ is sufficiently small or $\nu$ sufficiently large so that \eqref{eq:NSEVar} and \eqref{eq:Stokesh} have a unique solution, with sufficiently small norms. Then we have
\[
  \| \GRAD(\ue - \bu_h ) \|_{\bL^2(\omega,\Omega)} \leq C \left(  \inf_{\bw_h \in \bX_h} \| \GRAD( \ue - \bw_h) \|_{\bL^2(\omega,\Omega)} + \inf_{q_h \in M_h} \| \pe - q_h \|_{\bL^2(\omega,\Omega)} \right),
\]
where the constant $C$ may depend on $\bef$, $\nu$ and $\ue$, but is independent of $h$.
\end{theorem}
\begin{proof}
We split the difference $\ue - \bu_h = (\ue - S_h \ue) + (S_h \ue - \bu_h)$, where $S_h \ue$ is the velocity component of the Stokes projection of $(\ue,\pe)$. Owing to \cite[Corollary 4.2]{DuranOtarolaAJS} we have
\[
  \| \GRAD(\ue - S_h \ue ) \|_{\bL^2(\omega,\Omega)} \leq C \left( \inf_{\bw_h \in \bX_h} \| \GRAD( \ue - \bw_h) \|_{\bL^2(\omega,\Omega)} + \inf_{q_h \in M_h} \| \pe - q_h \|_{\bL^2(\omega,\Omega)} \right).
\]

Let $\be_h = S_h \ue - \bu_h$ and $\vare_h = S_h \pe - p_h$ and note that,
\[
  \begin{dcases}
    \int_\Omega \left(\nu\GRAD \be_h : \GRAD \bv_h - \vare_h \DIV \bv_h \right)= \int_\Omega \left( \ue \otimes \ue - \bu_h \otimes \bu_h \right): \bv_h  & \forall \bv_h \in \bX_h, \\
    \int_\Omega \DIV \be_h q_h = 0 & \forall q_h \in M_h.
  \end{dcases}
\]
The discrete stability of the Stokes projection shown in \cite[Theorem 4.1]{DuranOtarolaAJS} then implies
\begin{equation*}
  \| \GRAD \be_h \|_{\bL^2(\omega,\Omega)} + \| \vare_h \|_{L^2(\omega,\Omega)} \leq C_{4\to2}^2 \left( \| \GRAD \ue \|_{\bL^2(\omega,\Omega)} + \| \GRAD \bu_h \|_{\bL^2(\omega,\Omega)} \right)  \| \GRAD (\ue - \bu_h) \|_{\bL^2(\omega,\Omega)}.
\end{equation*}
We thus collect the derived estimates to arrive at
%Then, obviously,
\begin{multline*}
  \| \GRAD(\ue - \bu_h ) \|_{\bL^2(\omega,\Omega)} \leq C\left( \inf_{\bw_h \in \bX_h} \| \GRAD( \ue - \bw_h) \|_{\bL^2(\omega,\Omega)} + \inf_{q_h \in M_h} \| \pe - q_h \|_{\bL^2(\omega,\Omega)}  \right) \\ +
  C_{4\to2}^2 \left( \| \GRAD \ue \|_{\bL^2(\omega,\Omega)} + \| \GRAD \bu_h \|_{\bL^2(\omega,\Omega)} \right) \| \GRAD (\ue - \bu_h) \|_{\bL^2(\omega,\Omega)}.
\end{multline*}
The assumption that $\ue$ and $\bu_h$ are sufficiently small allow us to absorb the last term on the right hand side of this inequality into the left and conclude.
\qed\end{proof}

\end{document}